\documentclass[12pt]{article}
\usepackage{amsthm,amsmath,amssymb,enumerate}
\input xy
\xyoption{all}

\newtheorem{thm}{Theorem}[section]
\newtheorem{prop}[thm]{Proposition}
\newtheorem{lem}[thm]{Lemma}
\theoremstyle{definition}
\newtheorem*{rem}{Remark}
\newtheorem{ex}[thm]{Example}

\newcommand{\C}{\mathbb{C}}
\newcommand{\Z}{\mathbb{Z}}
\newcommand{\F}{\mathbb{F}}
\newcommand{\M}{\mathrm{M}}
\newcommand{\GL}{\mathrm{GL}}

\newcommand{\Irr}{\operatorname{Irr}}
\newcommand{\Ker}{\operatorname{Ker}}
\newcommand{\ol}{\overline}
\newcommand{\bs}{\boldsymbol}
\newcommand{\la}{\lambda}

\begin{document}
\title{Construction of rank $4$ self-dual association schemes inducing
  three partial geometric designs}
\author{Akihide Hanaki\footnote{Supported by JSPS KAKENHI Grant Number JP22K03266}}
\date{\small Faculty of Science, Shinshu University,
  Matsumoto, 390-8621, Japan\\
  {\tt hanaki@shinshu-u.ac.jp}
}
\maketitle

\begin{abstract}
  B. Xu characterized rank $4$ self-dual association schemes inducing
  three partial geometric designs by their character tables.
  We construct such association schemes as Schur rings over abelian $2$-groups.
  \medskip

\end{abstract}

\section{Introduction}
A block design is $(\mathfrak{P},\mathfrak{B},I)$,
where $\mathfrak{P}$ and $\mathfrak{B}$ are (finite) sets
and $I\subset \mathfrak{P}\times \mathfrak{B}$.
R. C. Bose, et al. defined partial geometric designs
in \cite{MR0505661}, and A. Neumaier called them $1\frac{1}{2}$-designs
in \cite{MR570206}.
A rank $(d+1)$ association scheme is $(X, \{R_i\}_{i=0,\dots,d})$,
where $X$ is a finite set, $R_i$ ($i=0,\dots,d$) are subsets of $X\times X$
with some properties (see Subsection \ref{secBA}).
Hence an association scheme contains many block designs for
$\mathfrak{P}=\mathfrak{B}=X$, $I=R_i$ (or a union of $R_i$'s).
In \cite{MR3570803,MR4522422,Xu2023},
partial geometric designs in association schemes were investigated.

In \cite{Xu2023}, Bangteng Xu characterized rank $4$
self-dual association schemes $(X, \{R_i\}_{i=0,\dots,3})$ such that
$R_1$, $R_2$, and $R_0\cup R_3$ induce partial geometric designs
by their character tables.
There are three cases : (1) primitive, (2) imprimitive symmetric,
and (3) imprimitive non-symmetric.
For the case (1), he gave infinitely many such examples
\cite[Theorem 9]{MR3570803}, \cite[Example 1.3]{Xu2023}.
but only one example for (2) and (3), respectively.
In this article, we will construct infinitely many examples for (2)
in Theorem \ref{thm1} by Schur rings over ${\Z_2}^{2n}$,
and (2) and (3) in Theorem \ref{thm2}
by Schur rings over ${\Z_4}^n$.

Xu's examples for (1) are of $3$-power order,
and our examples for (2) and (3) are of $2$-power order.
The author does not know whether there exists such an association scheme
whose order is not a prime power.

\section{Preliminaries}
For a positive integer $m$, we set $\Z_m=\Z/m\Z$.
A finite field of order $q$ will be denoted by $\F_q$. 

\subsection{Association schemes}\label{secBA}
Let $X$ be a finite set, and $R_i$ ($i=0,\dots,d$) nonempty subsets of $X\times X$.
We say that $\mathfrak{X}=(X, \{R_i\}_{i=0,\dots,d})$ is a rank $(d+1)$
(or $d$-class) \emph{association scheme} if
(1) $X\times X=\bigcup_{i=0,\dots,d}R_i$ and $R_i\cap R_j=\emptyset$ if $i\ne j$,
(2) $R_0=\{(x,x)\mid x\in X\}$,
(3) for every $i\in\{0,\dots,d\}$, there exists $i^*\in\{0,\dots,d\}$ such that
$R_{i^*}=\{(y,x)\mid (x,y)\in R_i\}$, and
(4) for $i, j, k\in \{0,\dots,d\}$, there exists an integer $p_{ij}^k$ such that
$p_{ij}^k=\sharp\{z\in X\mid (x,z)\in R_i,\ (z,y)\in R_j\}$ when $(x,y)\in R_k$.
An association scheme $\mathfrak{X}$ is said to be \emph{symmetric}
if $i^*=i$ for all $i\in\{0,\dots,d\}$,
\emph{commutative} if $p_{ij}^k=p_{ji}^k$ for all $i, j, k\in \{0,\dots,d\}$.
We set $n_i=p_{ii^*}^0$ and call this number the \emph{valency} of $R_i$.
It is known that $\mathfrak{X}$ is commutative if $d\leq 4$.
In this article, we only consider commutative association schemes.

We set $X=\{x_1,\dots,x_n\}$.
For $R_k$, we define the \emph{adjacency matrix} $A_k\in \M_n(\C)$ by
$(A_k)_{ij}=1$ if $(x_i,x_j)\in R_k$ and $0$ otherwise.
By the condition (4), $A_iA_j=\sum_{k=0}^d p_{ij}^k A_k$.
The $\C$-vector space $\C\mathfrak{X}=\bigoplus_{i=0}^d \C A_i$ becomes a
$\C$-subalgebra of the full matrix algebra $\M_n(\C)$.
By the condition (3), $\C\mathfrak{X}$ is a semisimple algebra.
Thus, for a commutative association scheme $\mathfrak{X}$,
$$\C\mathfrak{X}\cong \C\oplus\dots\oplus\C$$
as $\C$-algebras.
Every projection is an irreducible representation, and also an irreducible character
because the degree is one.
We denote by $\Irr(\mathfrak{X})$ the set of all irreducible characters.
The map $A_i\mapsto n_i$ is an irreducible character,
called the \emph{trivial character}.
The representation $A_i\mapsto A_i$ is called the \emph{standard representation},
and its character is called the \emph{standard character} and denoted by $\rho$.
It is clear that $\rho(A_0)=|X|$, $\rho(A_i)=0$ for $i\ne 0$.
Consider the irreducible decomposition
$\rho=\sum_{\chi\in\Irr(\mathfrak{X})}m_\chi \chi$, and we call $m_\chi$ the
\emph{multiplicity} of $\chi$.
In this article, we set $\Irr(\mathfrak{X})=\{\chi_0,\dots,\chi_d\}$ where
$\chi_0$ is the trivial character, and $m_i=m_{\chi_i}$.
The $(d+1)\times (d+1)$ matrix $(\chi_i(A_j))$ is called the \emph{character table}
(or the \emph{first eigenmatrix}) of $\mathfrak{X}$.

Two association schemes $\mathfrak{X}=(X, \{R_i\}_{i=0,\dots,d})$ and
$\mathfrak{X}'=(X', \{R_i'\}_{i=0,\dots,d'})$ are \emph{isomorphic}
if there is a bijection $\sigma:X\to X'$ such that
$\{R_i^\sigma\}_{i=0,\dots,d}=\{R_i'\}_{i=0,\dots,d'}$,
where $R_i^\sigma=\{(x^\sigma,y^\sigma)\mid (x,y)\in R_i\}$,
\emph{algebraically isomorphic}
if there is a bijection $\tau:\{0,\dots,d\}\to\{0,\dots,d'\}$ such that
$p_{i^\tau j^\tau}^{k^\tau}=p_{ij}^k$ for all $0\leq i, j, k\leq d$.
Obviously, if $\mathfrak{X}$ and $\mathfrak{X}'$ are isomorphic, then
they are algebraically isomorphic, but the converse is not true, in general.
Two commutative association schemes are algebraically isomorphic if and only if
their character tables are the same
by suitable permutations of their rows and columns.

Two orthogonality relations are known \cite[Theorem II.3.5]{BI}.

\begin{thm}
  Let $\mathfrak{X}=(X, \{R_i\}_{i=0,\dots,d})$ be a  commutative association scheme.
  \begin{enumerate}[{\rm (1)}]
    \item (The first orthogonality relation).
    $$\frac{m_i}{|X|}\sum_{j=0}^d \frac{1}{n_j}\chi_i(A_j)\ol{\chi_{i'}(A_j)}
    =\delta_{ii'},$$
    where $\ol{\chi_{i'}(A_j)}$ is the complex conjugate and $\delta_{ii'},$
    is the Kronecker's delta.
    \item (The second orthogonality relation).
    $$\frac{1}{n_j|X|}\sum_{i=0}^d m_i\chi_i(A_j)\ol{\chi_i(A_{j'})}=\delta_{jj'}.$$
  \end{enumerate}
\end{thm}

Let $\mathfrak{X}=(X, \{R_i\}_{i=0,\dots,d})$ be a  commutative association scheme.
A subset $C$ of $\{R_i\}_{i=0,\dots,d}$ is called a \emph{closed subset}
of $\mathfrak{X}$ if $\bigoplus_{R_i\in C}\C A_i$ is a subalgebra of $\C\mathfrak{X}$.
An association scheme is said to be \emph{primitive} if there is no non-trivial
closed subset, and \emph{imprimitive} otherwise.
For $\chi\in\Irr(\mathfrak{X})$, define the \emph{kernel} of $\chi$ by
$\Ker\chi=\{A_i\mid \chi(A_i)=n_i\}$.
Then $\Ker\chi$ is a closed subset of $\mathfrak{X}$.
The next lemma holds.

\begin{lem}\label{lemclosed}
  Let $\mathfrak{X}=(X, \{R_i\}_{i=0,\dots,d})$ be a commutative association scheme
  with a closed subset $C$, and let $\chi\in\Irr(\mathfrak{X})$.
  If $C\not\subset \Ker\chi$, then $\sum_{A_i\in C}\chi(A_i)=0$.
\end{lem}

A commutative association scheme $\mathfrak{X}=(X, \{R_i\}_{i=0,\dots,d})$
is said to be \emph{self-dual} if
$$\frac{\chi_i(A_j)}{n_j}=\frac{\chi_j(A_i)}{n_i}$$
hold for all $0\leq i, j\leq d$, by a suitable reordering of indices
\cite{Bannai1993}.

\subsection{Schur rings}\label{secBB}
Let $G$ be a finite group, and $\mathfrak{S}=\{S_0,\dots,S_d\}$ a partition of $G$.
Namely, $S_i$ ($i=0,\dots,d$) are non-empty subsets of $G$,
$G=S_0\cup\dots\cup S_d$, and $S_i\cap S_j=\emptyset$ if $i\ne j$.
For $S_i$, we define $\underline{S_i}=\sum_{g\in S_i}g\in \C G$
and $\C\mathfrak{S}=\bigoplus_{i=0}^d \C\underline{S_i}\subset \C G$.
We say that $\C\mathfrak{S}$ is a \emph{Schur ring} over $G$ if
(1) $S_0=\{1_G\}$, (2) for every $i\in\{0,\dots,d\}$, there exists
$i^*\in\{0,\dots,d\}$ such that $S_{i^*}=\{g^{-1}\mid g\in S_i\}$, and 
(3) $\C\mathfrak{S}$ is a subring of $\C G$ \cite{BI, MR0183775}.

\begin{ex}
  Let $G$ be a finite group, $H$ a subgroup of the automorphism group of $G$.
  Let $\mathfrak{S}=\{S_0,\dots,S_d\}$ be the set of $H$-orbits on $G$.
  Then $\C\mathfrak{S}$ is a Schur ring over $G$.
\end{ex}

A Schur ring defines an association scheme in a natural way.
Let $T:G\to \GL_n(\C)$ be the regular permutation representation of $G$,
where $n=|G|$.
Set $A_k=\sum_{g\in S_k}T(g)\in \M_n(\C)$.
In other words,
set $G=\{g_1,\dots,g_n\}$ and define $A_k\in \M_n(\C)$ by
$(A_k)_{ij}=1$ if $g_ig_j^{-1}\in S_k$ and $0$ otherwise.
Then $\{A_0,\dots,A_d\}$ becomes the set of adjacency matrices of an association scheme.

\subsection{Partial geometric designs, $1\frac{1}{2}$-designs}\label{secBC}
A block design is $(\mathfrak{P},\mathfrak{B},I)$,
where $\mathfrak{P}$ and $\mathfrak{B}$ are (finite) sets
and $I\subset \mathfrak{P}\times \mathfrak{B}$.

For a positive integer $t$,
a block design $(\mathfrak{P},\mathfrak{B},I)$ is a \emph{$t$-design}
if there is an integer $\lambda$ such that,
for any $t$-subset $S$ of $\mathfrak{P}$,
$\sharp\{B\in \mathfrak{B}\mid \text{$(p,B)\in I$ for all $p\in S$}\}=\lambda$.
A $1$-design $(\mathfrak{P},\mathfrak{B},I)$ is a \emph{partial geometric design}
or \emph{$1\frac{1}{2}$-design} if
there are integers $\alpha$, $\beta$ such that
$\sharp\{(y,C)\mid (x,C)\in I,\ (y,B)\in I\}=\alpha$ if $(x,B)\not\in I$
and $\beta$ if $(x,B)\in I$.
It is easy to see that $2$-designs are partial geometric designs.

Let $\mathfrak{X}=(X,\{R_i\}_{i=0,\dots,d})$ be an association scheme.
Since $R_i$ (or a union of them) is a subset of $X\times X$,
$R_i$ defines a block design $(X,X,R_i)$.
When $(X,X,R_i)$ is a (partial geometric) design, we say that
$R_i$ induces a (partial geometric) design.

\subsection{B. Xu's results}\label{secBD}
B. Xu characterized rank $4$
self-dual association schemes $\mathfrak{X}=(X, \{R_i\}_{i=0,\dots,3})$ such that
$R_1$, $R_2$, and $R_0\cup R_3$ induce partial geometric designs
by their character tables \cite[Theorems 1.2, 1.5, and 1.7]{Xu2023}.

\begin{thm}\label{thmXu}
  Let $\mathfrak{X}=(X, \{R_i\}_{i=0,\dots,3})$ be a self-dual association scheme
  such that $R_1$, $R_2$, and $R_0\cup R_3$ induce partial geometric designs.
  Then one of the following statements holds:
  \begin{enumerate}[(1)]
    \item $\mathfrak{X}$ is primitive and there exists an integer $\lambda$
    such that $3\mid \lambda$ and the character table of $\mathfrak{X}$ is
    $$T_1=\left(
      \begin{array}{cccc}
        1&\la(\la-1)&\la(\la+1)&(\la-1)(\la+1)\\
        1&\la&0&-\la-1\\
        1&0&-\la&\la-1\\
        1&-\la&\la&-1
      \end{array}
    \right).$$

    \item $\mathfrak{X}$ is imprimitive symmetric,
    and there exists an odd number $\theta$
    such that the character table of $\mathfrak{X}$ is
    $$T_2=\left(
      \begin{array}{cccc}
        1&\frac{\theta(\theta+1)}{2} & \frac{\theta(\theta+1)}{2} &\theta\\
        1&\frac{\theta+1}{2} & -\frac{\theta+1}{2} &-1\\
        1&-\frac{\theta+1}{2} & \frac{\theta+1}{2} &-1\\
        1&-\frac{\theta+1}{2} & -\frac{\theta+1}{2} &\theta\\
      \end{array}
    \right)$$
    or $\mathfrak{X}\cong K_m\wr K_2\wr K_m$ for some positive integer $m$. 
    \item $\mathfrak{X}$ is imprimitive non-symmetric,
    and there exists an odd number $\theta$
    such that the character table of $\mathfrak{X}$ is
    $$T_3=\left(
      \begin{array}{cccc}
        1&\frac{\theta(\theta+1)}{2} & \frac{\theta(\theta+1)}{2} &\theta\\
        1&\frac{(\theta+1)\sqrt{-1}}{2} & -\frac{(\theta+1)\sqrt{-1}}{2} &-1\\
        1&-\frac{(\theta+1)\sqrt{-1}}{2} & \frac{(\theta+1)\sqrt{-1}}{2} &-1\\
        1&-\frac{\theta+1}{2} & -\frac{\theta+1}{2} &\theta\\
      \end{array}
    \right)$$
    Conversely, if the character table of an association scheme is $T_1$,
    $T_2$, or $T_3$, then it is self-dual and
    $R_1$, $R_2$, and $R_0\cup R_3$ induce partial geometric designs.
\end{enumerate}
\end{thm}

Xu gave infinitely many examples in \cite[Example 1.3]{Xu2023}
for the case (1) for $\lambda=3^\ell$.
However, he gave only one example for (2) and (3) with $\theta=3$, respectively
\cite[Examples 1.6 and 1.8]{Xu2023}.

\section{Construction -- Schur rings over ${\Z_2}^{2n}$}
The method considered in \cite{Hanaki-Hirai-Ponomarenko2022} can be used
to construct association schemes having the character tables $T_2$ in
Theorem \ref{thmXu}.
  
Let $n$ be a positive integer, and consider
the finite field $\F_{2^n}$.
Let $V$ be a $2$-dimensional $\F_{2^n}$-vector space.
We set
$$\mathfrak{L}=\{L_a\mid a\in \F_{2^n}\cup\{\infty\}\},$$
where
\begin{eqnarray*}
  L_a &=& \{(x,ax) \mid x\in \F_{2^n}\}\quad (a\in \F_{2^n}),\\
  L_\infty&=&\{(0,x) \mid x\in \F_{2^n}\}.
\end{eqnarray*}
Remark that $|\mathfrak{L}|=2^n+1$.
We set $L_a^\sharp=L_a\setminus\{(0,0)\}$.
Decompose $\mathfrak{L}=P_1\cup P_2\cup P_3$ such that
$|P_1|=|P_2|=2^{n-1}$ and $|P_3|=1$.
We also set
\begin{eqnarray*}
  S_0 &=& \{(0,0)\},\\
  S_i&=& \bigcup_{L_a\in P_i} L_a^\sharp\quad (i=1,2,3).
\end{eqnarray*}          
Then the partition $S_0\cup S_1\cup S_2\cup S_3$ of $V$
defines a Schur ring over ${\Z_2}^{2n}$ ($\cong V$ as an abelian group)
\cite[Theorem 1]{Hanaki-Hirai-Ponomarenko2022}.
By the Schur ring, we can define a rank $4$ symmetric association scheme
$\mathfrak{X}=(X, \{R_i\}_{i=0,\dots,3})$.
We will determine the character table of $\mathfrak{X}$.

Since the association scheme $\mathfrak{X}$ is a fusion scheme of
an abelian group ${\Z_2}^{2n}$,
the adjacency algebra $\C\mathfrak{X}$ is a subalgebra of the group algebra
$\C{\Z_2}^{2n}$ and every irreducible character of $\mathfrak{X}$ is a restriction
of an irreducible character of ${\Z_2}^{2n}$.

\begin{prop}\label{prop3A}
  The character table of the association scheme $\mathfrak{X}$
  is $T_2$ for $\theta=2^n-1$ in Theorem \ref{thmXu}.
\end{prop}

we need an easy lemma.

\begin{lem}
  Let $L\in\mathfrak{L}$.
  If $\varphi\in\Irr(V)$, $\varphi\ne 1_V$, and $\Ker \varphi\supset L$,
  then $\sum_{\bs{v}\in L^\sharp}\varphi(\bs{v})=2^n-1$ and
  $\sum_{\bs{v}\in (L')^\sharp}\varphi(\bs{v})=-1$ for $L\ne L'\in\mathfrak{L}$.
\end{lem}

\begin{proof}
  By $\Ker \varphi\supset L$,
  $\sum_{\bs{v}\in L^\sharp}\varphi(\bs{v})=|L^\sharp|=2^n-1$.
  Suppose $L'\ne L$.
  If $\Ker \varphi\supset L'$, then $\varphi=1_V$.
  Thus $\Ker \varphi\not\supset L'$ and
  $\sum_{\bs{v}\in L'}\varphi(\bs{v})=0$.
\end{proof}

\begin{proof}[Proof of Proposition \ref{prop3A}]
  First of all, there is a trivial character $\chi_0$ of $\mathfrak{X}$.
  For any $L_a\in \mathfrak{L}$, there exists $1_V\ne \varphi\in\Irr(V)$
  such that $\Ker \varphi\supset L_a$.
  In this case, $\Ker \varphi\not\supset L_b$ for $b\ne a$.

  Suppose $P_3=\{L_a\}$ and choose $1_V\ne \varphi\in\Irr(V)$
  such that $\Ker \varphi\supset L_a$.
  Then $\varphi$ defines $\chi_3\in\Irr(\mathfrak{X})$ such that
  $\chi_3(A_0)=1$, $\chi_3(A_1)=\chi_3(A_2)=-2^{n-1}$, and $\chi_3(A_3)=2^n-1$.
  Suppose $L_a\in P_1$.
  Choose $1_V\ne \varphi\in\Irr(V)$
  such that $\Ker \varphi\supset L_a$.
  Then $\varphi$ defines $\chi_1\in\Irr(\mathfrak{X})$ such that
  $\chi_1(A_0)=1$, $\chi_1(A_1)=2^{n-1}$, $\chi_1(A_2)=-2^{n-1}$,
  and $\chi_1(A_3)=-1$.
  Similarly, we can define
  $\chi_2\in\Irr(\mathfrak{X})$ such that
  $\chi_2(A_0)=1$, $\chi_2(A_1)=-2^{n-1}$, $\chi_2(A_2)=2^{n-1}$,
  and $\chi_2(A_3)=-1$.
  Now we obtained $4$ distinct irreducible characters of $\mathfrak{X}$
  and the character table is $T_2$.
\end{proof}

\begin{thm}\label{thm1}
  The association scheme $\mathfrak{X}=(X, \{R_i\}_{i=0,\dots,3})$
  defined above is self-dual and
  $R_1$, and $R_2$, and $R_0\cup R_3$ induce partial geometric designs.
\end{thm}

\section{Construction -- Schur rings over ${\Z_4}^{n}$}
We write the matrix ring over $\Z_m$ of degree $n$ by $\M_n(\Z_m)$,
and the group consisting of invertible elements
of $\M_n(\Z_m)$ by $\GL_n(\Z_m)$.
By the usual matrix action,
$\GL_n(\Z_m)$ is just the automorphism group of an abelian group ${\Z_m}^n$.
The map $\overline{\bullet}:\Z_4\to \Z_2$ is the projection
$\Z_4\ni a+4\Z\mapsto a+2\Z\in \Z_2$.
The same notation will be used for ${\Z_m}^n$ and $\M_n(\Z_m)$.
The identity matrix in $\M_n(\Z_4)$ is denoted by $E_n$.

\begin{lem}\label{lemA}
  For $A\in \M_n(\Z_4)$, $A\in \GL_n(\Z_4)$ if and only if
  $\overline{A}\in \GL_n(\Z_2)$.
\end{lem}

Let $\F_{2^n}$ be the finite field of order $2^n$,
and let $\zeta$ be a primitive element of $\F_{2^n}$.
Since $\F_{2^n}=\{\sum_{i=0}^{n-1}a_i\zeta^i\mid a_i\in \F_2\ (i=0,\dots,n-1)\}$,
the map $h:\F_{2^n}\to{\Z_2}^n$, $\sum_{i=0}^{n-1}a_i\zeta^i \mapsto (a_0,\dots,a_{n-1})$
is a bijection.
Let $g:\F_{2^n}\to \M_n(\Z_2)$ be the regular representation of $\F_{2^n}$
as an $\F_2$-algebra.

\begin{lem}\label{lemB}
  There is a primitive element $\zeta$ of $\F_{2^n}$ and $P\in \M_n(\Z_4)$
  such that $\overline{P}$ is an image of $\zeta$ by the regular representation
  and the order of $P$ is $2^n-1$.
\end{lem}

\begin{proof}
  Fix a primitive element $\xi$ of $\F_{2^n}$ arbitrarily.
  Fix $Q\in M_n(\Z_4)$ such that $\overline{Q}=g(\xi)$.
  Since the order of $\overline{Q}$ is $2^n-1$,
  the order of $Q$ is a multiple of $2^n-1$.
  There exists $R\in \M_n(\Z_4)$ such that $Q^{2^n-1}=E_n+2R$.
  Thus $Q^{2(2^n-1)}=(E_n+2R)^2=E_n$.
  Now the order of $Q$ is $2^n-1$ or $2(2^n-1)$.
  If the order is $2^n-1$, then $\zeta=\xi$ and $P=Q$ satisfy the condition.

  Suppose that the order of $Q$ is $2(2^n-1)$.
  Since $2^n-1$ is odd, $\zeta=\xi^2$ is also a primitive element and
  the order of $P=Q^2$ is $2^n-1$.
  Now $\zeta$ and $P$ satisfy the required condition.
\end{proof}

Define $\overline{f}:{\Z_2}^n\to \M_n(\Z_2)$ by $\overline{f}=g\circ h^{-1}$.
Choose $P\in \M_n(\Z_4)$ such that $\overline{P}=g(\zeta)$
and the order of $P$ is $2^n-1$ by Lemma \ref{lemB}.
We may suppose that $P$ is of the following form :
$$P
=\left(
  \begin{array}{ccccc}
    0&1& \\
     &0&1&\\
     & & \ddots & \ddots &\\
     & &        & 0&1\\
    c_0&c_1&\dots&c_{n-1}&c_{n-1}
  \end{array}
\right).$$
Remark that the following arguments depends on the form of $P$.

\begin{ex}\label{ex3}
  Set $n=3$.
  The minimal polynomial of a primitive element of $\F_{2^3}$ is
  $x^3+x+1$. We put
  $$Q=\left(
    \begin{array}{ccc}
      0&1&0\\
      0&0&1\\
      1&1&0
    \end{array}
  \right).$$
  Since $Q^{7}\ne E_3$, we consider $Q^2$.
  Set $\bs{v}=(1,0,0)$.
  Then
  \begin{eqnarray*}
    \bs{v} &=& (1,0,0),\\
    \bs{v}Q^2 &=& (0,0,1),\\
    \bs{v}(Q^2)^2 &=& (0,1,1),\\
    \bs{v}(Q^2)^3 &=& (1,2,1)=1\bs{v}+3\bs{v}Q^2+2\bs{v}(Q^2)^2.
  \end{eqnarray*}
  We can choose
  $$P=\left(
    \begin{array}{ccc}
      0&1&0\\
      0&0&1\\
      1&3&2
    \end{array}
  \right)$$
  and $P^{7}=E_3$ holds.
\end{ex}

Define $f:{\Z_4}^n\to \M_n(\Z_4)$
by $f(b_0,\dots,b_{n-1})=\sum_{i=0}^{n-1}b_iP^i$.
Then the following diagram is commutative.

$$\xymatrix{
  {\Z_4}^n \ar[r]^{\overline{\bullet}} \ar[d]_f
  & {\Z_2}^n  \ar[d]_{\overline{f}} & \F_{2^n} \ar[l]_h \ar[dl]^g\\
  \M_n(\Z_4) \ar[r]_{\overline{\bullet}} & \M_n(\Z_2)
}$$

We set $H=\langle P\rangle$, a cyclic group of order $2^n-1$.

We set
\begin{eqnarray*}
  V_0&=&\{(0,\dots,0)\},\\
  V_1&=&\{(a_0,\dots,a_{n-1})\mid a_i\in \{0,2\}\ (i=0,\dots,n-1)\},\\
  V_2&=& {\Z_4}^n,
\end{eqnarray*}
and
$$V_1^\sharp = V_1\setminus V_0,\quad
  V_2^\sharp = V_2 \setminus V_1.$$
We set $K=\{E_n+f(\bs{v})\mid \bs{v}\in V_1\}$.
Since $(E_n+f(\bs{v}))(E_n+f(\bs{w}))=E_n+f(\bs{v}+\bs{w})$, 
$K$ is isomorphic to $V_1$, an elementary abelian group of order $2^n$.

\begin{lem}\label{lemC}
  The product $HK=H\times K$ is a group of order $2^n(2^n-1)$ and $V_0$, $V_1^\sharp$,
  $V_2^\sharp$ are $HK$-orbits on ${\Z_4}^n$.
  In particular, the action of $HK$ on $V_2^\sharp$ is regular.
\end{lem}

\begin{proof}
  Since every element of $K$ is a polynomial of $P$,
  $H$ commutes with $K$ element-wise.
  The orders of $H$ and $K$ are coprime, and we have $HK=H\times K$ is a group.

  Clearly $V_0$ is an orbit.
  Also $V_1^\sharp$ is an orbit because the action of $H$ on $V_1$ is same as
  the action of $\ol{H}$ on ${\F_2}^n$.

  Now we consider the stabilizer of $(1,0,\dots,0)$.
  Suppose that $(1,0,\dots,0)P^i(E_n+f(\bs{v}))=(1,0,\dots,0)$.
  Considering the image by $\ol{\bullet}$, $i$ must be $0$.
  Since $(1,0,\dots,0)(E_n+f(\bs{v}))=(1,0,\dots,0)+\bs{v}$
  by the form of $P$,
  we have $\bs{v}=\bs{0}$.
  So the stabilizer must be trivial.
  By $|HK|=2^n(2^n-1)=|V_2^\sharp|$, the action of $HK$ on $V_2^\sharp$ is regular.
\end{proof}

Let $W$ be a subgroup of $V_1$ of order $2^{n-1}$,
and set $K_0=\{E_n+f(\bs{w})\mid \bs{w}\in W\}$ and $G=HK_0$.

\begin{lem}\label{lemD}
  There are four $G$-orbits $S_0$, $S_1$, $S_2$, $S_3$ on ${\F_4}^n$ of length
  $1$, $2^{n-1}(2^n-1)$, $2^{n-1}(2^n-1)$, and $2^n-1$, respectively.
  If $(2,0,\dots,0)\in W$, then $-S_1=S_1$, and otherwise $-S_1=S_2$.
\end{lem}

\begin{proof}
  The first statement is clear by definition and Lemma \ref{lemC},
  where $S_0=V_0$, $S_1\cup S_2=V_2^\sharp$, and $S_3=V_1^\sharp$.

  Since $-S_1$ is also a $G$-orbit, $-S_1=S_1$ or $S_2$.
  Assume that $(1,0,\dots,0)\in S_1$.
  By $(1,0,\dots,0)(E_n+f(2,0,\dots,0))=-(1,0,\dots,0)$ and the regularity of the action
  of $HK$ on $V_2^\sharp$, the second statement holds.
  The similar thing holds for the case $(1,0,\dots,0)\in S_2$.
\end{proof}

Our group $G$ is a subgroup of $\GL_n(\Z_4)$,
the automorphism group of ${\Z_4}^n$.
The partition ${\Z_4}^n=S_0\cup S_1\cup S_2\cup S_3$ defines a Schur ring
and hence an association scheme $\mathfrak{X}$.
The association scheme $\mathfrak{X}$ is symmetric
if and only if $(2,0,\dots,0)\in W$.
We will determine the character table of $\mathfrak{X}$.

Since the association scheme $\mathfrak{X}$ is a fusion scheme of
an abelian group ${\Z_4}^n$,
every irreducible character of $\mathfrak{X}$ is a restriction
of an irreducible character of ${\Z_4}^n$.

\begin{prop}
  The character table of the association scheme $\mathfrak{X}$
  is $T_2$ if $\mathfrak{X}$ is symmetric,
  and $T_3$ if $\mathfrak{X}$ is non-symmetric.
  for $\theta=2^n-1$ in Section \ref{secBD}.
\end{prop}

\begin{proof}
  First of all, there is a trivial character $\chi_0$ of $\mathfrak{X}$.

  Let $\varphi$ be a non-trivial character of ${\Z_4}^n$
  not containing $V_1$ in its kernel.
  By Lemma \ref{lemclosed}, $\varphi(A_3)=-1$.
  We can set
  $S_1=\bs{e}HK_0=\{\bs{e}(E_n+f(\bs{w}))P^i\mid 0\leq i<2^n-1,\ \bs{w}\in W\},$
  where $\bs{e}=(1,0,\dots,0)$.
  Remark that $\bs{e}f(\bs{w})=\bs{w}$.
  We have
  \begin{eqnarray*}
    \varphi(A_1)&=&\sum_{i=0}^{2^n-2}\sum_{\bs{w}\in W}\varphi(\bs{e}(E+M(\bs{w}))P^i
                = \sum_{i=0}^{2^n-2}\sum_{\bs{w}\in W}
                    \varphi(\bs{e}P^i)\varphi(\bs{w}P^i)\\
                &=&\sum_{i=0}^{2^n-2}\left(\sum_{\bs{w}\in W}\varphi(\bs{w}P^i)\right)
                    \varphi(\bs{e}P^i).
  \end{eqnarray*}
  The group $\langle\ol{P}\rangle$ permutes
  ${\Z_2}^n\setminus\{(0,\dots,0)\}$ regularly,
  it permutes the set of all subgroups of ${\Z_2}^n$ of index $2$ regularly.
  Since $|\Ker \varphi\  \cap\  V_1|=2^{n-1}$,
  there is a unique $i$ such that $WP^i\subset \Ker\varphi$.
  Now $\sum_{\bs{w}\in W}\varphi(\bs{w}P^i)=2^{n-1}$ and
  $\sum_{\bs{w}\in W}\varphi(\bs{w}P^{i'})=0$ for $i'\ne i$.
  Thus $\varphi(A_1)=2^{n-1}\varphi(\bs{e}P^i)\in \{\pm 2^{n-1},\ \pm 2^{n-1}\sqrt{-1}\}$.
  There exists $\psi\in\Irr({\Z_4}^n)$ such that $\Ker\psi\supset V_1$ and
  $\psi(\bs{e}P^i)=-1$.
  Since $\psi\varphi\in\Irr({\Z_4}^n)$, $\Ker(\psi\varphi)\supset V_1$,
  and $(\psi\varphi)(\bs{e}P^i)=-\varphi(\bs{e}P^i)$,
  we may suppose that $\varphi(A_1)\in \{2^{n-1}, 2^{n-1}\sqrt{-1}\}$.
  By $\sum_{j=0}^4\varphi(A_j)=0$, $\varphi(A_2)=-\varphi(A_1)$.
  Now we determine $\chi_1,\chi_2\in\Irr(\mathfrak{X})$ by $\varphi,\psi\varphi$,
  respectively.

  If $\mathfrak{X}$ is symmetric, then $\chi_2(A_1)$ must be rational,
  and so $\chi_2(A_1)=2^{n-1}$.
  If $\mathfrak{X}$ is non-symmetric, then $\chi_2(A_1)$ must be irrational,
  and so $\chi_2(A_1)=2^{n-1}\sqrt{-1}$.

  The remaining irreducible character $\chi_3$ is determined
  by orthogonality relation.
  The character table is completely determined for each case.
\end{proof}

\begin{thm}\label{thm2}
  The association scheme $\mathfrak{X}=(X, \{R_i\}_{i=0,\dots,3})$
  defined above is self-dual and
  $R_0\cup R_3$, $R_1$, and $R_2$ define partial geometric designs.
\end{thm}

\begin{proof}
  The claim holds because the properties are determined by the character table.
\end{proof}

\begin{rem}
  The isomorphism class of $\mathfrak{X}$ depends on the choice of $W$.
  In Example \ref{ex3},
  the order of the automorphism groups of $\mathfrak{X}$
  are $5376$ for $W=\F_4(0,2,0)+\F_4(0,0,2)$
  and $1792$ for $W=\F_4(0,2,0)+\F_4(2,0,2)$.
\end{rem}

\bibliographystyle{amsplain}
\providecommand{\bysame}{\leavevmode\hbox to3em{\hrulefill}\thinspace}
\providecommand{\MR}{\relax\ifhmode\unskip\space\fi MR }
\providecommand{\MRhref}[2]{%
  \href{http://www.ams.org/mathscinet-getitem?mr=#1}{#2}
}
\providecommand{\href}[2]{#2}

\end{document}